\documentclass{amsart}
\usepackage{amsfonts}

\newtheorem{theorem}{Theorem}
\theoremstyle{plain}

\newtheorem{corollary}{Corollary}

\numberwithin{equation}{section}

\begin{document}
\title[Some Unified integrals of generalized K-Bessel function]{Certain
unified integral formulas involving the generalized modified k-Bessel
function of first kind}
\author{K. S. Nisar}
\address{Department of Mathematics, College of Arts \& Science-Wadi Addwaser%
\\
Prince Sattam bin Abdulaziz University, Saudi Arabia}
\email{ksnisar1@gmail.com}
\author{S.R. Mondal}
\address{Department of Mathematics, College of Science-Al Ahsa\\
King Faisal University, Saudi Arabia}
\email{saiful786@gmail.com}

\begin{abstract}
Generalized integral formulas involving the generalized modified k-Bessel
function $J_{k,\nu }^{c,\gamma ,\lambda }\left( z\right) $ of first kind are
expressed in terms generalized $k-$Wright functions .Some interesting
special cases of the main results are also discussed
\end{abstract}

\maketitle

\section{\protect\bigskip Introduction}

The integral formula involving various special functions have been studied
by many researchers (\cite{Brychkov},\cite{Choi1}). In 1888 Pincherle gave
the integrals involving product of Gamma functions along vertical lines (see 
\cite{Pincherle,Pincherle1,Pincherle2}). Barnes \cite{Barnes} , Mellin \cite%
{Mellin} and Cahen \cite{Cahen} extended some of these integrals in the
study of Riemann zeta function and other Drichlet's series. The integral
representation of Fox H-functions and hypergeometric $_{2}F_{1}$ functions
studied by \cite{Garg} and \cite{Ali} respectively. Also, the integral
representation of Bessel functions are given in many recent works (see \cite%
{Choi2}, \cite{Choi-Mathur}, \cite{Brychkov} and \cite{Watson}).

let $k\in R;\alpha ,\lambda ,\gamma ,\upsilon \in C;Re\left( \lambda \right)
>0,Re\left( \upsilon \right) >0,$the $k$-Bessel function of the first kind
defined by the following series \cite{Romero-Cerutti} :

\begin{equation}
J_{k,\nu }^{\left( \gamma \right) ,\left( \lambda \right) }\left( z\right)
=\sum_{n=0}^{\infty }\frac{\left( \gamma \right) _{n,k}}{\Gamma _{k}\left(
\lambda n+\upsilon +1\right) }\frac{\left( -1\right) ^{n}\left( z/2\right)
^{n}}{\left( n!\right) ^{2}}  \label{k1}
\end{equation}

where $\left( \gamma \right) _{n,k}$ is the $k-$Pochhammer symbol \ \cite%
{Diaz} is defined as:

\begin{equation}
\left( x\right) _{n,k}=x\left( x+k\right) \left( x+2k\right) ...\left(
x+\left( n-1\right) k\right) ,\gamma \in C,k\in R\text{ and }n\in N
\label{k2}
\end{equation}

and $\Gamma _{k}(z)$ is the $k-$gamma function, $k$ be the positive real
number, defined by (see \cite{Diaz})%
\begin{equation}
\Gamma _{k}\left( z\right) =\int_{0}^{\infty }e^{-\frac{t^{k}}{k}%
}t^{z-1}dt,Re\left( z\right) >0  \label{k3}
\end{equation}

Clearly, for $k=1$, $\Gamma _{k}(z)$ reduces to the classical $\Gamma \left(
z\right) $ function.

In this paper, we introduce a new generalization of $k-$Bessel function
called generalized modified $k-$Bessel function and is defined as:

let $k\in R;\alpha ,\lambda _{1},\gamma ,\upsilon ,c,b\in C;Re\left( \lambda
_{1}\right) >0,Re\left( \upsilon \right) >0,$the generalized modified $k-$%
Bessel function of the first kind given by the following series%
\begin{equation}
J_{k,\nu }^{c,\gamma ,\lambda }\left( z\right) =\sum_{n=0}^{\infty }\frac{%
\left( c\right) ^{n}\left( \gamma \right) _{n,k}}{\Gamma _{k}\left( \lambda
_{1}n+\upsilon +\frac{b+1}{2}\right) }\frac{\left( z/2\right) ^{\upsilon +2n}%
}{\left( n!\right) ^{2}}  \label{k4}
\end{equation}

The aim of this paper is to establish two generalized integral formulas,
which are expressed in terms of generalized $k-$Wright functions, by
inserting newly generalized modified k-Bessel function.

The generalized Wright hypergeometric function ${}_{p}\psi _{q}(z)$ is given
by the series 
\begin{equation}
{}_{p}\psi _{q}(z)={}_{p}\psi _{q}\left[ 
\begin{array}{c}
(a_{i},\alpha _{i})_{1,p} \\ 
(b_{j},\beta _{j})_{1,q}%
\end{array}%
\bigg|z\right] =\displaystyle\sum_{k=0}^{\infty }\dfrac{\prod_{i=1}^{p}%
\Gamma (a_{i}+\alpha _{i}k)}{\prod_{j=1}^{q}\Gamma (b_{j}+\beta _{j}k)}%
\dfrac{z^{k}}{k!},  \label{eqn-4-Struve}
\end{equation}%
where $a_{i},b_{j}\in \mathbb{C}$, and real $\alpha _{i},\beta _{j}\in 
\mathbb{R}$ ($i=1,2,\ldots ,p;j=1,2,\ldots ,q$). Asymptotic behavior of this
function for large values of argument of $z\in {\mathbb{C}}$ were studied in 
\cite{Foxc} and under the condition 
\begin{equation}
\displaystyle\sum_{j=1}^{q}\beta _{j}-\displaystyle\sum_{i=1}^{p}\alpha
_{i}>-1  \label{eqn-5-Struve}
\end{equation}%
was found in the work of \cite{Wright-2,Wright-3}. Properties of this
generalized Wright function were investigated in \cite{Kilbas}, (see also 
\cite{Kilbas-itsf, Kilbas-frac}. In particular, it was proved \cite{Kilbas}
that ${}_{p}\psi _{q}(z)$, $z\in {\mathbb{C}}$ is an entire function under
the condition ($\ref{eqn-5-Struve}$).

The generalized hypergeometric function represented as follows \cite%
{Rainville}:

\begin{equation}
_{p}F_{q}\left[ 
\begin{array}{c}
\left( \alpha _{p}\right) ; \\ 
\left( \beta _{q}\right) ;%
\end{array}%
z\right] =\sum\limits_{n=0}^{\infty }\frac{\Pi _{j=1}^{p}\left( \alpha
_{j}\right) _{n}}{\Pi _{j=1}^{p}\left( \beta _{j}\right) _{n}}\frac{z^{n}}{n!%
},  \label{eqn-1-hyper}
\end{equation}

provided $p\leq q; p=q+1$ and $\left\vert z\right\vert <1$

where $\left( \lambda \right) _{n}$ is well known Pochhammer symbol defined
for $\left( \text{ for }\lambda \in C\right) $ (see \cite{Rainville})

\begin{equation}
\left( \lambda \right) _{n}:=\left\{ 
\begin{array}{c}
1\text{ \ \ \ \ \ \ \ \ \ \ \ \ \ \ \ \ \ \ \ \ \ \ \ \ \ \ \ \ \ \ \ \ \ \
\ \ }\left( n=0\right) \\ 
\lambda \left( \lambda +1\right) ....\left( \lambda +n-1\right) \text{ \ \ \
\ \ \ \ \ \ \ \ \ \ }\left( n\in N:=\{1,2,3....\}\right)%
\end{array}%
\right.  \label{eqn-2-hyper}
\end{equation}

\begin{equation}
\left( \lambda \right) _{n}=\frac{\Gamma \left( \lambda +n\right) }{\Gamma
\left( \lambda \right) }\text{ \ \ \ \ \ \ \ \ \ \ }\left( \lambda \in
C\backslash Z_{0}^{-}\right) .  \label{eqn-2b-hyper}
\end{equation}

where $Z_{0}^{-}$ is the set of non-positive integers.

If we put $\alpha _{1}=...=\alpha _{p}=\beta _{1}=....=\beta _{q} $ in ($\ref%
{eqn-4-Struve}$),then ($\ref{eqn-1-hyper}$) is a special case of the
generalized Wright function:

\begin{equation}
{}_{p}\psi _{q}(z)={}_{p}\psi _{q}\left[ 
\begin{array}{c}
\left( \alpha _{1},1\right) ,...,\left( \alpha _{p},1\right) ; \\ 
\left( \beta _{1},1\right) ,...,\left( \beta _{q},1\right) ;%
\end{array}%
z\right] =\dfrac{\prod_{j=1}^{p}\Gamma (\alpha _{j})}{\prod_{j=1}^{q}\Gamma
(\beta _{j})}\text{ }_{p}F_{q}\left[ 
\begin{array}{c}
\alpha _{1},...,\alpha _{p}; \\ 
\beta _{1},...,\beta _{q};%
\end{array}%
z\right]  \label{eqn-3-hyper}
\end{equation}

For the present investigation, we need the following result of Oberbettinger 
\cite{Ober}

\begin{equation}
\int_{0}^{\infty }x^{\mu -1}\left( x+a+\sqrt{x^{2}+2ax}\right) ^{-\lambda
}dx=2\lambda a^{-\lambda }\left( \frac{a}{2}\right) ^{\mu }\frac{\Gamma
\left( 2\mu \right) \Gamma \left( \lambda -\mu \right) }{\Gamma \left(
1+\lambda +\mu \right) }  \label{eqn-int1}
\end{equation}

provided $0<Re\left( \mu \right) <Re\left( \lambda \right) $

Also, we need the following relation of $\Gamma _{k}$ with the classical
gamma Euler function (see \cite{Cerutti}:

\begin{equation}
\Gamma _{k}\left( z+k\right) =z\Gamma _{k}\left( z\right)  \label{k5}
\end{equation}

\begin{equation}
\Gamma _{k}\left( z\right) =k^{\frac{z}{k}-1}\Gamma \left( \frac{z}{k}\right)
\label{k6}
\end{equation}

\begin{equation}
\Gamma _{k}\left( k\right) =1  \label{k7}
\end{equation}

\section{Main results}

Two generalized integral formulas established here, which expressed in terms
of generalized $k-$Wright functions $\left( \ref{eqn-3-hyper}\right) $ by
inserting the generalized modified k-Bessel function of the first kind $%
\left( \ref{k4}\right) $ with the suitable argument in the integrand of $%
\left( \ref{eqn-int1}\right) $

\begin{theorem}
For $\lambda ,\mu ,\nu ,c,\lambda _{1}\in \mathbb{C},$ $Re\left( \lambda
+\nu +2\right) >Re\left( \mu \right) >0$ and $x>0$. Then the following
formula holds true:%
\begin{eqnarray}
&&\int_{0}^{\infty }x^{\mu -1}\left( x+a+\sqrt{x^{2}+2ax}\right) ^{-\lambda
}J_{k,\upsilon }^{c,\gamma ,\lambda _{1}}\left( \frac{y}{x+a+\sqrt{x^{2}+2ax}%
}\right) dx  \notag \\
&=&2^{1-\nu -\mu }a^{\mu -\lambda -\upsilon }y^{\upsilon }k^{-2\mu }\Gamma
\left( 2\mu \right)   \notag \\
&&\times _{k,2}\Psi _{3}\left[ 
\begin{array}{c}
\left( \lambda +\upsilon +k,2\right) ,\left( k\left( \nu +\lambda -\mu
\right) ,2k\right) ; \\ 
\left( \nu +\frac{b+1}{2},\lambda _{1}\right) ,\left( k\left( 1+\lambda
+\upsilon +\mu \right) ,2k\right) ,\left( \lambda +\nu ,2\right) ;%
\end{array}%
\frac{cy^{2}}{4a^{2}}\right]   \label{1}
\end{eqnarray}

where $_{k,2}\Psi _{3}$ denote the k-Fox-Wright function \cite{Cerutti}
\end{theorem}

\begin{proof}
By applying $\left( \ref{k4}\right) $ to the LHS of $\left( \ref{1}\right) $
and interchanging the order of integration and summation, which is verified
by uniform convergence of the involved series under the given conditions ,
we obtain

\begin{eqnarray*}
&&\int_{0}^{\infty }x^{\mu -1}\left( x+a+\sqrt{x^{2}+2ax}\right) ^{-\lambda
}J_{k,\upsilon }^{c,\gamma ,\lambda _{1}}\left( \frac{y}{x+a+\sqrt{x^{2}+2ax}%
}\right) dx \\
&=&\int_{0}^{\infty }x^{\mu -1}\left( x+a+\sqrt{x^{2}+2ax}\right) ^{-\lambda
} \\
&&\times \sum_{n=0}^{\infty }\frac{\left( c\right) ^{n}\left( \gamma \right)
_{n,k}}{\Gamma _{k}\left( \lambda _{1}n+\upsilon +\frac{b+1}{2}\right) }%
\frac{\left( \frac{y}{2}\right) ^{\upsilon +2n}}{\left( n!\right) ^{2}}%
\left( x+a+\sqrt{x^{2}+2ax}\right) ^{-\left( \upsilon +2n\right) }dx \\
&=&\sum_{n=0}^{\infty }\frac{\left( c\right) ^{n}\left( \gamma \right) _{n,k}%
}{\Gamma _{k}\left( \lambda _{1}n+\upsilon +\frac{b+1}{2}\right) }\frac{%
\left( \frac{y}{2}\right) ^{\upsilon +2n}}{\left( n!\right) ^{2}}%
\int_{0}^{\infty }x^{\mu -1}\left( x+a+\sqrt{x^{2}+2ax}\right) ^{-\left(
\lambda +\nu +2n\right) }dx
\end{eqnarray*}

In view of the conditions given in Theorem 1, since $\mathbb{R}\left(
\lambda +\nu \right) >\mathbb{R}\left( \mu \right) >0$ $k\in N_{0}:=N\cup
\left\{ 0\right\} .$

Applying $\left( \ref{eqn-int1}\right) $ to the integrand of $\left( \ref{1}%
\right) $\ and obtain the following expression:%
\begin{eqnarray*}
&=&\sum_{n=0}^{\infty }\frac{\left( c\right) ^{n}\left( \gamma \right) _{n,k}%
}{\Gamma _{k}\left( \lambda _{1}n+\upsilon +1\right) }\frac{\left( \frac{y}{2%
}\right) ^{\upsilon +2n}}{\left( n!\right) ^{2}}2\left( \lambda +\nu
+2n\right) a^{-\left( \lambda +\nu +2n\right) }\left( \frac{a}{2}\right)
^{\mu } \\
&&\times \frac{\Gamma \left( 2\mu \right) \Gamma \left( \lambda +\nu +2n-\mu
\right) }{\Gamma \left( 1+\lambda +\nu +\mu +2n\right) }
\end{eqnarray*}

By making the use of the relation $\left( \ref{k5}\right) $, we obtain

\begin{eqnarray*}
&=&2^{1-\nu -\mu }a^{\mu -\lambda -\upsilon }y^{\upsilon }k^{-2\mu }\Gamma
\left( 2\mu \right)  \\
&&\times \sum_{k-0}^{\infty }\frac{\Gamma _{k}\left( \lambda +\nu
+k+2n\right) \Gamma _{k}\left( \lambda k+\nu k-\mu k+2kn\right) }{\Gamma
_{k}\left( \lambda +\nu 2n\right) \Gamma _{k}\left( \lambda _{1}n+\upsilon +%
\frac{b+1}{2}\right) \Gamma _{k}\left( k+\lambda k+\nu k+\mu k+2kn\right) }
\\
&&\times \frac{\left( c\right) ^{n}}{\left( n!\right) ^{2}}\left( \frac{y^{2}%
}{4a^{2}}\right) ^{n}
\end{eqnarray*}

which is the desired result.
\end{proof}

\begin{corollary}
Let the conditions of Theorem 1 be satisfied and let $k=\lambda _{1}=1$ and $%
c=-c$ in $\left( \ref{1}\right) $. Then the following integral formula holds:%
\begin{eqnarray*}
&&\int_{0}^{\infty }x^{\mu -1}\left( x+a+\sqrt{x^{2}+2ax}\right) ^{-\lambda
}J_{k,\upsilon }^{c,\gamma ,\lambda _{1}}\left( \frac{y}{x+a+\sqrt{x^{2}+2ax}%
}\right) dx \\
&=&2^{1-\nu -\mu }a^{\mu -\lambda -\upsilon }y^{\upsilon }\Gamma \left( 2\mu
\right)  \\
&&\times _{2}\Psi _{3}\left[ 
\begin{array}{c}
\left( 1+\lambda +\upsilon ,2\right) ,(\nu +\lambda -\mu ,2); \\ 
\left( \nu +\frac{b+1}{2},\lambda _{1}\right) ,\left( 1+\lambda +\upsilon
+\mu ,2\right) ,\left( \lambda +\nu ,2\right) ;%
\end{array}%
\frac{-cy^{2}}{4a^{2}}\right] 
\end{eqnarray*}

which is the result given by \cite{Choi-Mathur} .
\end{corollary}

\begin{corollary}
Setting $b=c=1$ in $\left( \ref{1}\right) $ with some appropriate parameter
replacements, we get the integral formula of Bessel function $J_{\upsilon
}\left( z\right) $ given by Choi and Agarwal \cite{Choi2}.
\end{corollary}

\begin{theorem}
For $\lambda _{1},\mu ,\nu ,c,b\in \mathbb{C},0<\mathbb{R}\left( \mu +\nu
+2\right) <\mathbb{R}\left( \lambda +\nu +2\right) ~$and $x>0,$ then the
following integral formula holds true:%
\begin{eqnarray}
&&\int_{0}^{\infty }x^{\mu -1}\left( x+a+\sqrt{x^{2}+2ax}\right) ^{-\lambda
}J_{k,\upsilon }^{c,\gamma ,\lambda _{1}}\left( \frac{xy}{x+a+\sqrt{x^{2}+2ax%
}}\right) dx  \notag \\
&=&2^{1-2\nu -\mu }y^{\nu }a^{\mu -\lambda }k^{1+\lambda -\mu }\Gamma \left(
\lambda -\mu \right)  \notag \\
&&\times _{k,2}\Psi _{3}\left[ 
\begin{array}{c}
\left( k\left( 2\mu +2\nu \right) ,4k\right) ,\left( \nu +\lambda
+k,2\right) ; \\ 
\left( \nu +1,\lambda _{1}\right) ,\left( \nu +\lambda ,2\right) ,\left(
k\left( 1+\lambda +\mu +2\nu \right) ,4k\right) ;%
\end{array}%
\frac{cy^{2}}{4}\right]  \label{2}
\end{eqnarray}

where $_{k,2}\Psi _{3}$ denote the k-Fox-Wright function \cite{Cerutti}
\end{theorem}

\begin{proof}
By applying $\left( \ref{k4}\right) $ to the LHS of $\left( \ref{2}\right) $
and interchanging the order of integration and summation, which is verified
by uniform convergence of the involved series under the given conditions ,
we obtain

\begin{eqnarray*}
&&\int_{0}^{\infty }x^{\mu -1}\left( x+a+\sqrt{x^{2}+2ax}\right) ^{-\lambda
}J_{k,\upsilon }^{c,\gamma ,\lambda _{1}}\left( \frac{xy}{x+a+\sqrt{x^{2}+2ax%
}}\right) dx \\
&=&\int_{0}^{\infty }x^{\mu -1}\left( x+a+\sqrt{x^{2}+2ax}\right) ^{-\lambda
} \\
&&\times \sum_{n=0}^{\infty }\frac{\left( c\right) ^{n}\left( \gamma \right)
_{n,k}}{\Gamma _{k}\left( \lambda _{1}n+\upsilon +\frac{b+1}{2}\right) }%
\frac{\left( \frac{xy}{2}\right) ^{\upsilon +2n}}{\left( n!\right) ^{2}}%
\left( x+a+\sqrt{x^{2}+2ax}\right) ^{-\left( \upsilon +2n\right) }dx \\
&=&\sum_{n=0}^{\infty }\frac{\left( c\right) ^{n}\left( \gamma \right) _{n,k}%
}{\Gamma _{k}\left( \lambda _{1}n+\upsilon +\frac{b+1}{2}\right) }\frac{%
\left( \frac{y}{2}\right) ^{\upsilon +2n}}{\left( n!\right) ^{2}} \\
&&\times \int_{0}^{\infty }x^{\left( \mu +\upsilon +2n\right) -1}\left( x+a+%
\sqrt{x^{2}+2ax}\right) ^{-\left( \lambda +\upsilon +2n\right) }dx
\end{eqnarray*}

Applying $\left( \ref{eqn-int1}\right) $ to the integrand of $\left( \ref{2}%
\right) $\ ,we obtain the following expression:%
\begin{eqnarray*}
&=&\sum_{n=0}^{\infty }\frac{\left( c\right) ^{n}\left( \gamma \right) _{n,k}%
}{\Gamma _{k}\left( \lambda _{1}n+\upsilon +\frac{b+1}{2}\right) }\frac{%
\left( \frac{y}{2}\right) ^{\upsilon +2n}}{\left( n!\right) ^{2}}2\left(
\lambda +\upsilon +2n\right) a^{-\left( \lambda +\upsilon +2n\right) }\left( 
\frac{a}{2}\right) ^{\mu +\upsilon +2n} \\
&&\times \frac{\Gamma \left( 2\mu +2\nu +4n\right) \Gamma \left( \lambda
-\mu \right) }{\Gamma \left( 1+\lambda +2\nu +\mu +4n\right) }
\end{eqnarray*}

By making the use of $\left( \ref{k5}\right) $, we obtain%
\begin{eqnarray*}
&=&2^{1-2\nu -\mu }y^{\nu }a^{\mu -\lambda }k^{\lambda -\mu +1}\Gamma \left(
\lambda -\mu \right)  \\
&&\times \sum_{k-0}^{\infty }\frac{\left( \gamma \right) _{n,k}\Gamma
_{k}\left( \lambda +\nu +k+2n\right) \Gamma _{k}\left( 2\mu k+2\nu
k+4nk\right) }{\left( n!\right) ^{2}\Gamma _{k}\left( \lambda _{1}n+\upsilon
+\frac{b+1}{2}\right) \Gamma _{k}\left( \lambda +\nu +2n\right) \Gamma
_{k}\left( \left( 1+\lambda +\mu +2\upsilon +4n\right) k\right) }\left( 
\frac{cy^{2}}{4}\right) ^{n}
\end{eqnarray*}

which is the desired result.
\end{proof}

\begin{corollary}
Let the conditions given in Theorem 2 satisfied and set $k=\lambda _{1}=1,$%
and $c=-c$ \ Theorem 2 reduces to

\begin{eqnarray*}
&&\int_{0}^{\infty }x^{\mu -1}\left( x+a+\sqrt{x^{2}+2ax}\right) ^{-\lambda
}J_{1,\upsilon }^{c,\gamma ,1}\left( \frac{xy}{x+a+\sqrt{x^{2}+2ax}}\right)
dx \\
&=&2^{1-2\nu -\mu }y^{\nu }a^{\mu -\lambda }\Gamma \left( \lambda -\mu
\right)  \\
&&\times _{2}\Psi _{3}\left[ 
\begin{array}{c}
\left( 2\mu +2\nu ,4\right) ,\left( \nu +\lambda +1,2\right) ; \\ 
\left( \nu +\frac{b+1}{2},\lambda _{1}\right) ,\left( \nu +\lambda ,2\right)
,\left( 1+\lambda +\mu +2\nu ,4\right) ;%
\end{array}%
\frac{cy^{2}}{4}\right] 
\end{eqnarray*}%
which is the result given by \cite{Choi-Mathur} .
\end{corollary}

\begin{corollary}
setting $b=c=1$ in $\left( \ref{1}\right) $ with some appropriate parameter
replacements, we get the integral formula of Bessel function $J_{\upsilon
}\left( z\right) $ given by Choi and Agarwal \cite{Choi2}.
\end{corollary}

\textbf{Conclusion}\newline
The integral formulas for generalized modified $k-$Bessel function of first
kind is derived and the results are expressed interm of generalized $k-$%
Wright function.Some of interesting special cases also derived from the main
results. Using some suitable parametric replacement, theorems 1 and 2 gives
the unified integral representation of generalized Bessel function, if $c=-1$
and integral representation of modified Bessel function by , if $c=1$ .


\end{document}